\documentclass[smallextended]{svjour3}
\usepackage{graphicx,amsmath,amsfonts}
\usepackage{apacite}
\usepackage{color,soul}
\setstcolor{red}

\usepackage{natbib}

\smartqed

\newcommand{\cn}{{\mathcal N}}

\newcommand{\cd}{{\mathcal D}}
\newcommand{\ce}{{\mathcal E}}
\newcommand{\cl}{{\mathcal L}}

\def\row#1#2{{#1}_1,\ldots ,{#1}_{#2}}

\def\row#1#2{{#1}_1,\ldots ,{#1}_{#2}}

\journalname{Order}
\title{Generalisation of the Danilov-Karzanov-Koshevoy Construction for Peak-Pit Condorcet Domains}
\titlerunning{Peak-Pit Condorcet Domains}

\author{Arkadii Slinko}

\institute{
           Arkadii Slinko \at
           Department of Mathematics,
           University of Auckland, 
           Auckland, New Zealand
           \email{a.slinko@auckland.ac.nz}
}

\begin{document}

\maketitle

\begin{abstract}
Danilov, Karzanov and Koshevoy (2012) geometrically introduced an interesting operation of composition on tiling Condorcet domains and using it they disproved a long-standing problem of Fishburn about the maximal size of connected Condorcet domains. We give an algebraic definition of this operation and investigate its properties. We give a precise formula for the cardinality of  composition of two Condorcet domains and improve the Danilov, Karzanov and Koshevoy result showing that Fishburn's alternating scheme does not always define a largest peak-pit Condorcet domain.
\end{abstract}

\section{Introduction}

The famous Condorcet Paradox shows that if voters' preferences are unrestricted, the majority voting can lead to intransitive collective preference in which case the Condorcet Majority Rule \citep{Condorcet1785}, despite all its numerous advantages, is unable to determine the best alternative, i.e., it is not always decisive. 
Domain restrictions is, therefore, an important topic in economics and computer science alike \citep{elkind2018restricted}.  
In particular, for artificial societies of autonomous software agents there is no problem of individual freedom and, hence, for the sake of having transitive collective decisions the designers can restrict choices of those artificial agents in order to make the majority rule work every time. 

Condorcet domains represent a solution to this problem, they are sets of linear orders with the property that, whenever the preferences of all voters belong to this set, the majority relation of any profile with an odd number of voters is transitive. Maximal Condorcet domains historically have attracted a special attention since they represent a compromise which  allows a society to always have transitive collective preferences and, under this constraint,  provide voters with as much individual freedom as possible. The question: ``How large a Condorcet domain can be?'' has attracted even more attention  (see the survey of \citet{mon:survey} for a fascinating account of historical developments). \citet{kim1992overview} identified this problem as a major unsolved problem in the mathematical social sciences. \citet{PF:1996} introduced the function
\[
f(n)=\max \{|\cd| : \text{$\cd$ is a Condorcet domain on the set of $n$ alternatives.}\}
\]
and put this problem in the mathematical perspective asking for maximal values of this function. 

\citet{Abello91} and \citet{PF:1996,PF:2002} managed to construct some ``large'' Condorcet domains based on different ideas. Fishburn, in particular, taking a clue from Monjardet example (sent to him in private communication), came up with the so-called alternating scheme domains (that will be defined later in the text), later called Fishburn's domains \citep{DKK:2012}. This scheme produced Condorcet domains with some nice properties, which, in particular, are connected and have maximal width (see the definitions of these concepts later in this paper). 
\citet{PF:1996} conjectured (Conjecture 2) that among Condorcet domains that do not satisfy the so-called never-middle condition (these in \cite{DKK:2012} were later called peak-pit domains), the alternating scheme provides domains of maximum cardinality. \citet{GR:2008} formulated another similar hypothesis (Conjecture 1) which later appeared to be equivalent to Fishburn's one \citep{DKK:2012}. \citet{monjardet2006condorcet} introduced the function
\[
g(n)=\max \{|\cd| : \text{$\cd$ is a peak-pit domain on the set of $n$ alternatives}\}
\]
in terms of which Fishburn's hypothesis becomes $g(n)=|F_n|$, where $F_n$ is the $n$th Fishburn domain. 
\citet{mon:survey} also emphasised Fishburn's hypothesis.   

It is known that $g(n)=f(n)$ for $n\le 7$  \citep{PF:1996, GR:2008} and it is believed that $g(16)<f(16)$ \citep{mon:survey}. This is because \cite{PF:1996} showed that $f(16)>|F_{16}|$. Thus, if Fishburn's hypothesis were true we would get $f(n)>g(n)$ for large $n$. However, this hypothesis is not true.

\citet{DKK:2012} introduced the class of tiling domains which are peak-pit domains of maximal width and defined an operation on tiling domains that allowed them to show that $g(42)>|F_{42}|$. This operation was somewhat informally defined which made investigation of it and application of it in other situations difficult. In the present article we give an algebraic definition and a generalisation of the Danilov-Karzanov-Koshevoy construction and investigate its properties. 
In our interpretation it involves two peak-pit Condorcet domains $\cd_1$ and $\cd_2$ on sets of $n$ and $m$ alternatives, respectively, and two linear orders $u\in\cd_1$ and $v\in \cd_2$; the result is denoted as $(\cd_1\otimes \cd_2)(u,v)$. It is again a peak-pit Condorcet domain on $n+m$ alternatives whose exact cardinality we can calculate. Using this formula we can slightly refine the argument from \cite{DKK:2012} to show that $g(40)>|F_{40}|$.


\section{Preliminaries}
\label{sec:prel}

Let $A$ be a finite set and $\mathcal{L}(A)$ be the set of all (strict) linear orders on $A$. Any subset $\cd \subseteq \mathcal{L}(A)$ will be called a {\em domain}. Any sequence $P=(\row vn)$ of linear orders from $\cd$ will be called  a {\em profile} over $\cd$\footnote{A profile, unlike the domain, can have several identical linear orders.}. It usually represents a collective set of opinions of a society about merits of alternatives from $A$. A linear order $a_1>a_2>\cdots>a_n$ on $A$, will be denoted by a string $a_1a_2\ldots a_n$.  Let us also introduce notation for reversing orders: if $x=a_1a_2\ldots a_n$, then $\bar{x}=a_na_{n-1}\ldots a_1$. If linear order $v_i$ ranks $a$ higher than $b$, we denote this as $a\succ_ib$.

\begin{definition}
The {\em majority relation} $\succeq_P$ of a profile $P$ is defined as
\[
a\succeq_P b \Longleftrightarrow |\{i\mid a\succ_i b\}| \ge  |\{i\mid b\succ_i a\}|.
\]
Verbally,  $a\succeq_P b$ means that at lest as many voters from a society with profile $P$ prefer $a$ to $b$ as voters who prefer $b$ to $a$. 
For an odd number of linear orders in the profile $P$ this relation is a tournament, i.e., complete and asymmetric binary relation. In this case we denote it $\succ_P$.
\end{definition}

Now we can define the main object of this investigation.

\begin{definition}
A domain $\cd \subseteq \cl(A)$ over a set of alternatives $A$ is a {\em Condorcet domain} if the majority relation $\succ_P$ of any profile $P$ over $\cd $ with odd number of voters is transitive. A Condorcet domain $\cd$ is {\em maximal} if for any Condorcet domain ${\cd}' \subseteq \cl(A)$ the inclusion $\cd \subseteq {\cd}'$ implies $\cd={\cd}'$. 
\end{definition}

There is a number of alternative definitions of Condorcet domains, see e.g., \citet{mon:survey,puppe2019condorcet}.

Up to an isomorphism, there is only one maximal Condorcet domain on the set $\{a,b\}$, namely $CD_2=\{ab, ba\}$ and there are only three maximal Condorcet domains on the set of alternatives $\{a,b,c\}$, namely, 
\begin{align*}
CD_{3,t}&= \{abc, acb, cab, cba\}, \quad    CD_{3,m}=\{abc, bca, acb, cba\}, \\    
&\hspace*{15mm}CD_{3,b}= \{abc, bac, bca, cba\}.
\end{align*}
The first domain contains all the linear orders on $a,b,c$ where $b$ is never ranked first, second contains all the linear orders on $a,b,c$ where $a$ is never ranked second and the third contains all the linear orders on $a,b,c$ where $b$ is never ranked last. Following Monjardet, we denote these conditions as $bN_{\{a,b,c\}}1$, $aN_{\{a,b,c\}}2$ and $bN_{\{a,b,c\}}3$,  respectively.  We note that these are the only conditions of type $xN_{\{a,b,c\}}i$ with $x\in \{a,b,c\}$ and $i\in \{1,2,3\}$ that these domains satisfy.  

A domain that for any triple $a,b,c\in A$ satisfies a condition $xN_{\{a,b,c\}}1$ with $x\in \{a,b,c\}$ is called {\em never-top} domain, a domain that for any triple $a,b,c\in A$ satisfies a condition $xN_{\{a,b,c\}}2$ with $x\in \{a,b,c\}$ is called {\em never-middle} domain, and a domain that for any triple $a,b,c\in A$ satisfies a condition $xN_{\{a,b,c\}}3$ with $x\in \{a,b,c\}$ is called {\em never-bottom} domain. 

\begin{definition}[\cite{DKK:2012}]
A domain that for any triple satisfies either never-top or never-bottom condition is called a {\em peak-pit domain}. Both never-top and never-bottom conditions will be called {\em peak-pit conditions}. 
\end{definition}

We note that \citet{DKK:2012}, who consider linear orders over $A=\{1,2,\ldots, n\}$, restrict in their investigation the class of peak-pit domains to domains that contain two completely reversed orders (up to an isomorphism they can be taken as $12\ldots n$ and $\overline{12\ldots n}=n\,n-1\ldots 1$) and prove that under this restriction all of them can be embedded into tiling domains (Theorem~2 of \cite{DKK:2012}).  We also note that never-bottom domains are also known as Arrow's single-peaked domains and maximal domains among them have all cardinality $2^{n-1}$ \citep{slinko2019}.\par\medskip

Given a set of alternatives $A$, we say that 
\begin{equation}
\label{complete_set}
\cn=\{ xN_{\{a,b,c\}}i \mid \{a,b,c\}\subseteq A,\   x\in \{a,b,c\}\ \text{and $i\in \{1,2,3\}$} \}
\end{equation}
is a {\em complete set of never conditions} if it contains at least one never condition for every triple $a,b,c$ of distinct elements of~$A$. If the set of linear orders that satisfy $\cn$ is non-empty, we say that $\cn$ is consistent.

\begin{proposition}
\label{classic}
A domain of linear orders $\cd\subseteq \mathcal{L}(A)$ is a Condorcet domain if and only if it is non-empty and satisfies a complete set of never conditions. 
\end{proposition}

\begin{proof}
This is well-known characterisation noticed by many researchers. See, for example, Theorem~1(d)  in \cite{puppe2019condorcet} and references there.
\end{proof}

This proposition, in particular, means that the collection $\cd(\mathcal{N})$ of all linear orders that satisfy a certain complete set of never conditions  $\mathcal{N}$, if non-empty, is a Condorcet domain. Let us also denote by $\mathcal{N}(\cd)$ the set of all never conditions that are satisfied by all linear orders from a domain~$\cd$. 

Let $\psi\colon A\to A'$ be a bijection between two sets of alternatives. It can then be extended to a mapping $\psi\colon \cl(A)\to \cl(A')$ in two ways: by mapping a  linear order $u=a_1a_2\ldots a_m$ onto $\psi(u)=\psi(a_1)\psi(a_2)\ldots \psi(a_m)$\footnote{We use the same notation for both mappings since there can be no confusion.} or to $ \overline{\psi (u)}=\psi(a_m)\psi(a_{m-1})\ldots \psi(a_1).$

\begin{definition}
Let $A$ and $A'$ be two sets of alternatives (not necessarily distinct) of equal cardinality.
We say that two domains, $\cd 
\subseteq \mathcal{L}(A)$ and $\cd'
\subseteq \mathcal{L}(A')$ 
 are {\em isomorphic} if there is a bijection $\psi\colon A\to A'$ such that $\cd'=\{\psi(d) \mid d\in \cd\}$
and flip-isomorphic if $\cd'=\{\overline{\psi(d)} \mid d\in \cd\}$. 
\end{definition}

\begin{example}
\label{cd1_and_2}
The single-peaked and single-dipped maximal Condorcet domains on $\{a,b,c\}$ are $CD_{3,b}= \{abc, bac, bca, cba\}$   and $CD_{3,t}= \{abc, acb, cab, cba\}$, respectively.
They are not isomorphic but flip-isomorphic under the identity mapping of $\{a,b,c\}$ onto itself.
\end{example}

\begin{definition}[\cite{Puppe2018}]
\label{def:max_width}
A Condorcet domain $\cd$ is said to have {\em maximal width} if it contains two completely reversed orders, i.e., together with some linear order $u$ it also contains its flip $\bar{u}$.
\end{definition}

Up to an isomorphism, for any Condorcet domain $\cd$ of maximal width we may assume that $A=\{1,2,\ldots, n\}$ and it contains linear orders $e=12\ldots n$ and $\bar{e}=n\ldots 21$. \par\medskip

The universal domain $\cl (A)$ is naturally endowed with  the following betweenness structure (as defined by \cite{Kemeny1959}).  
An order $v$ is {\em between} orders $u$ and $w$ if
$v\supseteq u\cap w$, i.e.,~$v$ agrees with all binary comparisons in which
$u$ and $w$ agree (see also \cite{kem-sne:b:polsci:mathematical-models}). The set of all
orders that are between $u$ and $w$ is called the {\em interval}
spanned by $u$ and $w$ and is denoted by $[u,w]$. The domain
$\cl (A)$ endowed with this betweenness relation is referred to as the
{\em permutahedron} \citep{mon:survey}.\par\medskip

Given a domain of preferences $\cd$, for any $u,w\in \cd$ we define the induced interval as $[u,w]_\cd=[u,w]\cap \cd$.  \citet{puppe2019condorcet} defined a graph $G_\cd$ associated with this domain. The set of linear orders from $\cd$ are the set of vertices  $V_\cd$ of $G_\cd$, and for two orders $u,w\in \cd$ we draw an edge between them if there is no other vertex between them, i.e., $[u,w]_\cd=\{u,w\}$. The set of edges is denoted $E_\cd$ so the graph is $G_\cd=(V_\cd,E_\cd)$. As established in \cite{puppe2019condorcet}, for any Condorcet domain $\cd$ the graph $G_\cd$  is a median graph \citep{mulder} and any median graph can be obtained in this way. 

A domain $\cd$ is called {\em connected} if its graph $G_\cd$ is a subgraph of the permutahedron \citep{puppe2019condorcet}; we note that domains $CD_{3,t}$ and $CD_{3,b}$ are connected but $CD_{3,m}$ is not.  \citet{DKK:2012} called a domain of maximal width {\em semi-connected} if the two completely reversed orders can be connected by a path of vertices that is also a path in the permutahedron corresponding to a maximal chain in the Bruhat order. They proved that a maximal Condorcet domain of maximal width is semi-connected if and only if it is a peak-pit domain. \citet{Puppe2017} showed that for a maximal Condorcet domain semi-connectedness implies direct connectedness (Proposition~A2) which means that any two linear orders in the domain are connected by a shortest possible (geodesic) path.\par\medskip

Finally, we give two more definitions that express two properties of Condorcet domains. But, firstly, we will introduce the following notation. Suppose $\cd\subseteq \cl(A)$ be a domain on the set $A$ and let $B\subseteq A$. Suppose also $u\in \cd$. Then by $\cd_B$ and $u_B$ we denote the restrictions of $\cd$ and $u$ onto $B$, respectively. 

\begin{definition}
\label{def:copious}
We call a Condorcet domain $\cd$ {\em ample} if for any pair of alternatives $a,b\in A$ the restriction $\cd_{\{a,b\}}$ of this domain to $\{a,b\}$ has two distinct orders, that is, $\cd_{\{a,b\}}=\{ab,ba\}$.
\end{definition}

A rhombus tiling (or simply a tiling) is a subdivision $T$ into rhombic tiles of the zonogon $Z(n; 2)$ obtained as the Minkowski sum of $n$ segments $s_i=[0,\psi_i]$, $i=1,\ldots,n$. This centre-symmetric $2n$-gon has the bottom vertex $b=(0,0)$ and the top vertex $t=s_1+\ldots+s_n$. A snake is a path from $b$ to $t$ which, for each $i=1,\ldots,n$ contains a unique segments parallel to $s_i$. Each snake corresponds to a linear order on $\{1,\ldots,n\}$ in the following way. If a point traveling from $b$ to $t$ passes segments parallel to $s_{i_1},s_{i_2}\ldots, s_{i_n}$, then the corresponding linear order will be $i_1i_2\ldots i_n$. The set of snakes of a rhombus tiling, thus, defines a domain which is called {\em tiling domain}. \cite{DKK:2012} showed that peak-pit domains of maximal width are exactly the tiling domains (see an example on Figure~\ref{F4-passport}).

\begin{definition}[\cite{slinko2019}]
\label{def:copious}
A Condorcet domain $\cd$ is called {\em copious} if for any triple of alternatives $a,b,c\in A$ the restriction $\cd_{\{a,b,c\}}$ of this domain to this triple has four distinct orders, that is, $|\cd_{\{a,b,c\}}|=4$.
\end{definition}

Of course, any copious Condorcet domain is ample. We note that, if a domain $\cd $ is copious, then it satisfies a unique set of never conditions \eqref{complete_set}. 

\begin{definition}
A complete set of peak-pit conditions~\eqref{complete_set} is said to satisfy the {\em alternating scheme} \citep{PF:1996}, if
for all $1\le i<j<k\le n$ it includes
\[
\text{ $jN_{\{i,j,k\}}3$, if j is even, and $ jN_{\{i,j,k\}}1$, if $j$ is odd}
\]
or 
\[
\text{ $jN_{\{i,j,k\}}1$, if j is even, and $jN_{\{i,j,k\}}3$, if $j$ is odd}.
\]
\end{definition}

The domains that are determined by these complete sets we define $F_n$ and $\overline{F_n}$, respectively, and call Fishburn's domains \citep{DKK:2012}. The second domain is flip-isomorphic to the first so we consider only the first one. 

In particular, $F_2=\{12,21\}$, $F_3=\{123, 213, 231, 321\}$ and $$F_4=\{ 1234, 1243, 2134, 2143, 2413, 2431, 4213, 4231, 4321 \}.$$

\begin{figure}[t]
\begin{center}
\includegraphics[height=6.5cm]{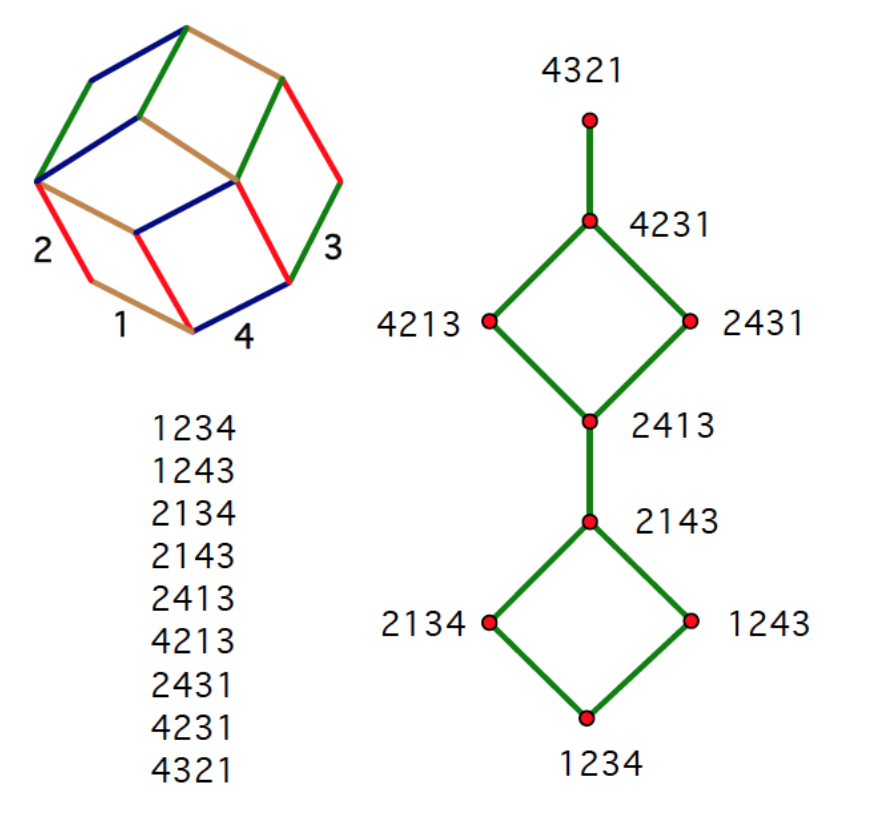}
\end{center}
\caption{Median graph and the tiling of the hexagon for Fishburn's domain $F_4$}
\label{F4-passport}
\end{figure}

Figure~\ref{F4-passport} shows the median graph of $F_4$ and its representation as a tiling domain. 

\cite{GR:2008} give the exact formula for the cardinality of $F_n$:
\begin{equation}
\label{card_alt_s}
|F_n|= (n+3)2^{n-3} - 
\begin{cases}
(n-\frac{3}{2}){n-2\choose \frac{n}{2}-1} & \text{for even $n$}\\
(\frac{n-1}{2}){n-1\choose \frac{n-1}{2}} & \text{for odd $n$}
\end{cases}
\end{equation}

Given a path in the permutahedron from $e=12\ldots n$ to $\bar{e}=nn{-}1\ldots 1$ where each pair $(i,j)$ with $1\le i<j\le n$ is switched exactly once (which can be associated with a maximal chain in the Bruhat order $\mathbb{B}(n,1)$) we say that it satisfies the {\em inversion triple} $[i,j,k]$ with $i<j<k$ if The pairs in this triple are switched in the order $(j,k),(i,k),(i,j)$. 
\cite{GR:2008} showed that if a maximal Condorcet domain $\cd$  of maximal width contains one maximal chain in the Bruhat order $\mathbb{B}(n,1)$ (i.e., is semi-connected), then it is a union of all equivalent maximal chains, i.e., those chains that satisfy the same set of inversion triples. Thus any maximal semi-connected Condorcet domain $\cd$ can be defined by the set of inversion triples. In particular, the domain $F_4$ can be defined by the set of inversion triples
\[
\{[1,3,4], [2,3,4]\}.
\]

\section{Main Results}

Let us start with an observation.

\begin{proposition}
\label{max_width_cop}
Let $\cd$ be a semi-connected Condorcet domain of maximal width on the set of alternatives $A$. Then: 
\begin{itemize}
\item[(i)] For any $a\in A$ its restriction $\cd'$ on $A'=A-\{a\}$ is also a semi-connected domain of maximal width.
\item[(ii)] $\cd$ is copious peak-pit domain.
\end{itemize}
\end{proposition}

\begin{proof}
(i) If $w$ and $\bar{w}$ are two completely reversed linear orders in $\cd$, then after removal of $a$, their images will still be completely reversed. 
Let $u,v$ be two vertices in $G_\cd$ which are neighbouring vertices in the permutahedron on the path connecting $w$ and $\bar{w}$. Then $v$ differs from $u$ by a swap of neighbouring alternatives. Let $u',v'$ be their images under the natural mapping of $\cd$ onto $\cd'$. If one of these swapped alternatives was $a$, then $u'=v'$. If not, $u',v'$ will still differ by a swap of neighbouring alternatives. Hence $\cd'$ is semi-connected. \par\smallskip

(ii)  Let $a,b,c\in A$ and let $\cd''$ be the restriction of $\cd$ onto $\{a,b,c\}$. Since $\cd$ is of maximal width, the same can be said about $\cd''$ and without loss of generality we may assume that $\cd''$ contains $abc$ and $cba$. By (i) $\cd''$ is semi-connected and hence there will be two intermediate orders in $\cd''$ connecting $abc$ and $cba$. These would be either $acb$ and $cab$ or $bac$ and $bca$. Thus, $\cd''$ has four linear orders, and, hence, $\cd$ is copious domain satisfying $bN_{\{a,b,c\}}1$ or $bN_{\{a,b,c\}}3$, respectively. Hence it is a peak-pit domain. 
\end{proof}

\subsection{Danilov-Karzanov-Koshevoy construction and its generalisation}

\begin{figure}
\begin{center}
\includegraphics[height=4cm]{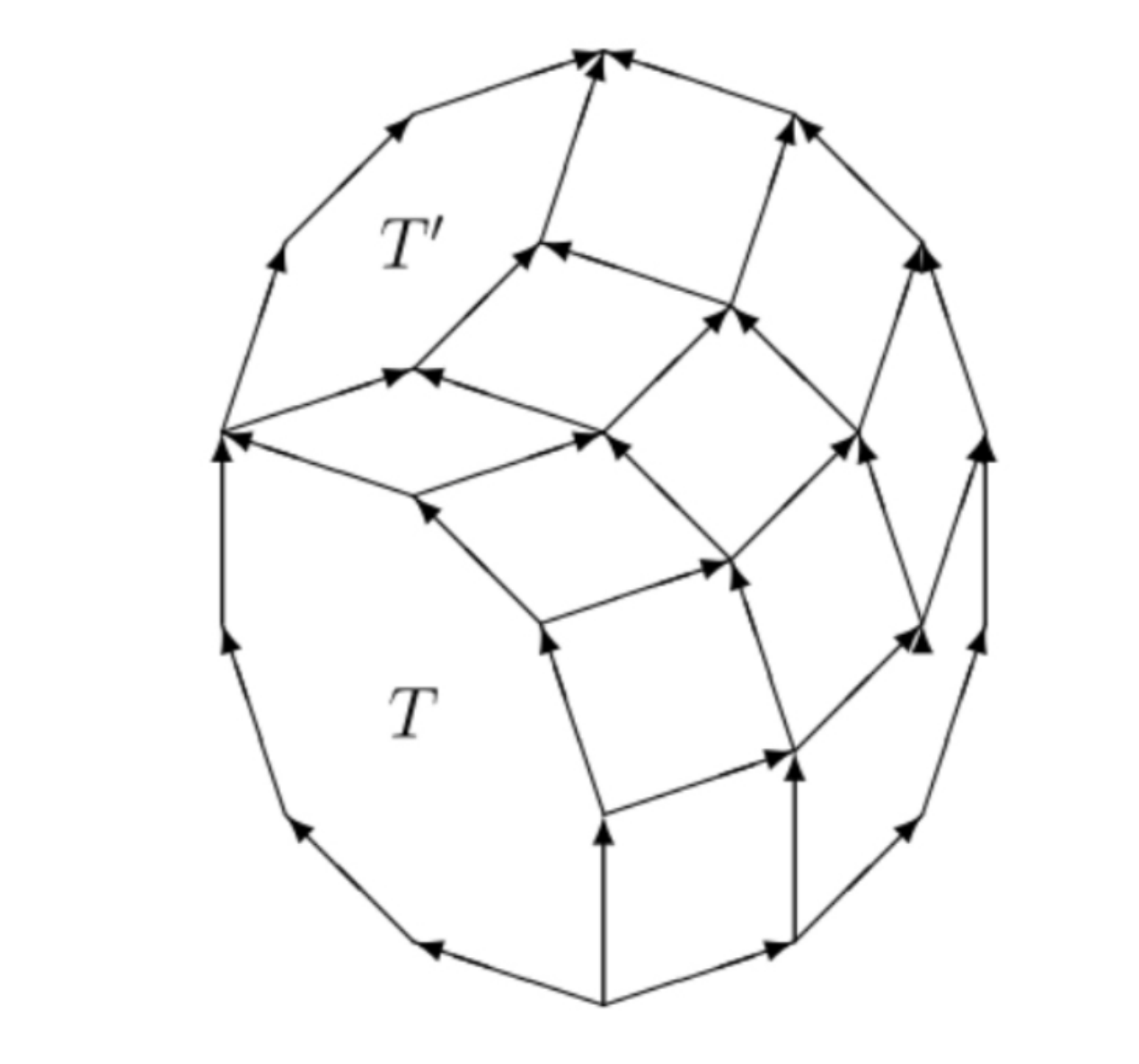}
\end{center}
\caption{Concatenation of tilings $T$ and $T'$}
\label{concat}
\end{figure}

Danilov-Karzanov-Koshevoy \citep{DKK:2012} define the `concatenation' of two tiling domains by the picture shown in Figure~\ref{concat} (where one arrow is obviously missing).

Let us now start describing this construction algebraically. In fact, this will be a generalisation of their construction since in our construction two arbitrary linear orders are involved. Firstly, we describe `pure' concatenation.

Let $\cd_1$ and $\cd_2$ be two Condorcet domains on disjoint sets of alternatives $A$ and $B$, respectively. We define a {\em concatenation} of these domains as the domain
\[
\cd_1\odot \cd_2 = \{xy \mid x\in \cd_1 \ \text{and $y\in \cd_2$}\}
\]
on $A\cup B$. It is immediately clear that $\cd_1\odot \cd_2$ is also a Condorcet domain of cardinality $|\cd_1\odot \cd_2|=|\cd_1||\cd_2|$. We have only to check that one of the never-conditions is satisfied for triples $\{a_1,a_2,b\}$ where $a_1,a_2\in A$ and $b\in B$ (for triples $\{a,b_1,b_2\}$ the argument will be similar).  The restriction $(\cd_1\odot \cd_2)|_{\{a_1,a_2,b\}}$ will contain at most two linear orders $a_1a_2b$ and $a_2a_1b$, which is consistent with both never-top and never-bottom conditions. This domain corresponds to $T$ and $T'$ on Figure~\ref{concat}. 

\begin{definition}
Let $A$ and $B$ be two disjoint sets of alternatives, $u\in \cl(A)$ and $v\in\cl(B)$. An order $w\in \cl (A\cup B)$ is said to be a {\em shuffle} of $u$ and $v$ if $w_A=u$ and $w_B=v$, i.e.,
the restriction of $w$ onto $A$ is equal to $u$ and the restriction of $w$ onto $B$ is equal to $v$. 
\end{definition}

For example,  $516723849$ is a shuffle of $1234$ and $56789$.\par\medskip
 
Given two linear orders $u$ and $v$, we define domain $u\oplus v$ as the set of all shuffles of $u$ and~$v$. It is clear from definition that $u\oplus v=v\oplus u$. The cardinality of this domain is $|u\oplus v|={n+m\choose m}$.  We believe this domain corresponds to what is depicted in Figure~\ref{concat} outside of $T$ and $T'$.

Now we combine the two domains together.

\begin{theorem}
Let $\cd_1$ and $\cd_2$ be two Condorcet domains on disjoint sets of alternatives $A$ and $B$. Let $u\in \cd_1$ and $v\in \cd_2$ be arbitrary linear orders. Then
\[
(\cd_1\otimes \cd_2)(u,v) := (\cd_1\odot \cd_2) \cup (u\oplus v)
\]
is a Condorcet domain. Moreover,  if $\cd_1$ and $\cd_2$ are peak-pit domains, so is $(\cd_1\otimes \cd_2)(u,v)$. 
\end{theorem}

\begin{proof}
Let us fix $u$ and $v$ in this construction and denote $(\cd_1\otimes \cd_2)(u,v)$ as simply $\cd_1\otimes \cd_2$.
 If $a,b,c\in A$, then $(\cd_1\otimes \cd_2)_{\{a,b,c\}}=(\cd_1)_{\{a,b,c\}}$, i.e., the restriction of $\cd_1\otimes \cd_2$ onto $\{a,b,c\}$ is the same as the restriction of $\cd_1$ onto $\{a,b,c\}$. Hence $\cd_1\otimes \cd_2$ satisfies the same never condition for $\{a,b,c\}$ as $\cd_1$. For $x,y,z\in B$ the same thing happens. 

Suppose now $a,b\in A$ and $x\in B$. Then $(\cd_1\odot \cd_2)_{\{a,b,x\}}\subseteq \{abx, bax\}$. Let also $u_{\{a,b\}}=\{ab\}$. Then 
$(u\oplus v)_{\{a,b,x\}}=\{abx, axb, xab\}$, hence 
\begin{equation}
\label{abx}
(\cd_1\otimes \cd_2)_{\{a,b,x\}}\subseteq \{abx, bax, axb, xab\}, 
\end{equation}
thus $\cd_1\otimes \cd_2$ satisfies $aN_{\{a,b,x\}}3$. For $a\in A$ and $x,y\in B$ we have $(\cd_1\odot \cd_2)_{\{a,x,y\}}\subseteq \{axy, ayx\}$. Let also $v_{\{x,y\}}=\{xy\}$. Then 
$(u\oplus v)_{\{a,x,y\}}=\{axy, xay, xya\}$, hence 
\begin{equation}
\label{axy}
(\cd_1\otimes \cd_2)_{\{a,x,y\}}\subseteq \{axy, ayx, xay, xya\}, 
\end{equation}
 thus $\cd_1\otimes \cd_2$ satisfies $yN_{\{a,x,y\}}1$.
\end{proof}

{\bf Note:} The inequalities \eqref{abx} and \eqref{axy} become equalities if for any $i\in \{1,2\}$ and any $a,b\in \cd_i$ we have $(\cd_i)_{\{a,b\}} =\{ab,ba\}$, i.e., if $\cd_1$ and $\cd_2$ are ample.  

\begin{proposition}
\label{cardinality}
If $|A|=m$ and $|B|=n$, then for any $u\in\cd_1$ and $v\in \cd_2$
\begin{equation}
\label{cardinality_of_tensor_product}
|(\cd_1\otimes \cd_2)(u,v)|=|\cd_1||\cd_2|+{n+m\choose m}-1.
\end{equation}
\end{proposition}

\begin{proof}
We have $|\cd_1\otimes \cd_2|=|\cd_1||\cd_2|$ and $|u\oplus v|={n+m\choose m}$. These two sets have only one linear order in common which is $uv$. This proves \eqref{cardinality_of_tensor_product}.
\end{proof}

\begin{proposition}
\label{max_width}
Let $\cd_1$ and $\cd_2$ be of maximal width with $u,\bar{u}\in \cd_1$ and $v,\bar{v}\in \cd_2$.  Then $(\cd_1\otimes \cd_2)(u,v)$ is also of maximal width. If $\cd_1$ and $\cd_2$ are semi-connected, then so is $(\cd_1\otimes \cd_2)(u,v)$. 
\end{proposition}

\begin{proof}
Since $\cd_1$ and $\cd_2$ are of maximal width, we have $\bar{u}\in \cd_1$ and $\bar{v}\in\cd_2$. Hence $\bar{u}\bar{v}\in \cd_1\odot \cd_2$. We also have $vu \in u\oplus v$, and $\overline{vu}=\bar{u}\bar{v}$, hence $(\cd_1\otimes \cd_2)(u,v)$ has maximal width. To prove the last statement we note that $\bar{u}\bar{v}$ can be connected to $uv$ (which belongs both to $\cd_1\otimes \cd_2$ and to $u\oplus v$) by a geodesic path and $uv$ in turn can be connected to $vu$ by a geodesic path within $u\oplus v$.
\end{proof}

If both $\cd_1$ and $\cd_2$ have maximal width, it is not true, however, that $(\cd_1\otimes \cd_2)(u,v)$ will have maximal width for any $u\in\cd_1$ and $v\in \cd_2$. 
 Let us take, for example, $\cd_1=\{x=ab,\bar{x}=ba\}$ and $\cd_2=\{u=cde, v=dec, w=dce, \bar{u}=edc\}$. Then $(\cd_1\otimes \cd_2)(x,u)$ has maximal width while $(\cd_1\otimes \cd_2)(x,v)$ does not since $\bar{v}\notin \cd_2$. In particular,
\[
(\cd_1\otimes \cd_2)(x,u)\not\cong (\cd_1\otimes \cd_2)(x,v).
\]
This indicates that the construction of the tensor product may be useful in description of Condorcet domains which do not satisfy the requirement of maximal width.

\begin{proposition}
\label{cor:i-ii}
Let $\cd_1$ and $\cd_2$ be two Condorcet domains on disjoint sets of alternatives $A$ and $B$. Let $u\in \cd_1$ and $v\in \cd_2$ be arbitrary linear orders. Then
\begin{itemize}
\item[(i)] $(\cd_1\otimes \cd_2)(u,v)$ is connected, whenever $\cd_1$ and $\cd_2$ are;
\item[(ii)] $(\cd_1\otimes \cd_2)(u,v)$ is copious, whenever $\cd_1$ and $\cd_2$ are.
\end{itemize}
\end{proposition}

\begin{proof}
(i) If $\cd = (\cd_1\otimes \cd_2)(u,v)$ is connected, then $\cd_1$ and $\cd_2$ are connected too. Suppose now that $\cd_1$ and $\cd_2$ are connected, suppose $w,w'\in \cd$ which are neighbours in $\Gamma_\cd$. Since all neighbours in $\cd_1\otimes \cd_2$ are neighbours in the permutahedron and so are neighbours in $\cd_1\oplus \cd_2$, it is enough to consider the case when $w\in \cd_1\odot \cd_2$ and $w'\in \cd_1\oplus \cd_2$.  But $uv$ is on the shortest path from $w$ to $w'$ and it is in $\cd$. Hence either $w=uv$ or $w'=uv$ and either $\{w,w'\}\subseteq \cd_1\otimes \cd_2$ or $\{w,w'\}\subseteq \cd_1\oplus \cd_2$. This proves (i).

(ii) This part follows from \eqref{abx} and \eqref{axy} since, as was noted before, when $\cd_1$ and $\cd_2$ are copious these inequalities become equalities.
\end{proof}

\begin{proposition}
The following isomorphism holds
\begin{equation}
\label{product_of_CD2}
(F_2(a,b)\otimes F_2(c,d))(ab,cd) \cong F_4(b,a,d,c).
\end{equation}
\end{proposition}

\begin{proof}
We list orders of this domain as columns of the following matrix
\[
[F_2(a,b)\odot F_2(c,d) \mid ab\oplus cd]=
\left[\begin{array}{cccc|ccccccccccc}
a&a&b&b&a&a&c&c&c   \\
b&b&a&a&c&c&a&a&d   \\
c&d&c&d&b&d&b&d&a   \\
d&c&d&c&d&b&d&b&b   
\end{array}
\right]. 
\]
We see that the following never conditions are satisfied: $aN_{\{a,b,c\}}3$, $aN_{\{a,b,d\}}3$, $dN_{\{a,c,d\}}1$, $dN_{\{b,c,d\}}1$. Hence the mapping $1\to b$, $2\to a$, $3\to d$ and $4\to c$ is an isomorphism of $F_4$ onto the tensor product $(F_2(a,b)\otimes F_2(c,d))(ab,cd)$.
\end{proof}

The isomorphism \eqref{product_of_CD2} is very nice but unfortunately for larger $m,n$ we have $F_m\otimes F_n\not\cong F_{m+n}$. Moreover, it appears that for two maximal Condorcet domains $\cd_1$ and $\cd_2$ on sets $A$ and $B$, respectively, $\cd_1\otimes \cd_2$ may not be maximal on $A\cup B$. Here is an example. 

\begin{example}
Let us calculate $\ce:=F_3(1,2,3)\otimes F_2(4,5)(321,54)$:
\[
\left[\begin{array}{cccccccc|cccccccccccccccccccccccccc}
1&2&2&3&1&2&2&3&3&3&5&3&3&3&5&5&5\\
2&1&3&2&2&1&3&2&2&5&3&2&5&5&3&3&4\\
3&3&1&1&3&3&1&1&5&2&2&5&2&4&2&4&3\\
4&4&4&4&5&5&5&5&1&1&1&4&4&2&4&2&2\\
5&5&5&5&4&4&4&4&4&4&4&1&1&1&1&1&1
\end{array}\right].
\]
There are 17 linear orders in this domain. It is  known, however, that $F_5$ has 20  \citep{PF:1996} but this fact alone does not mean non-maximality of $\ce$. By Proposition~\ref{cor:i-ii} this domain is copious. By its construction it satisfies just three inversion triples:
\[
[1,2,4],\quad [1,3,4],\quad [2,3,4].
\]
Now we see that there are two more linear orders $23514$ and $23541$ that satisfy these conditions. Hence $\ce$ is not maximal.
\end{example}

\subsection{On Fishburn's hypothesis}

We will further write $(F_k\otimes F_m)(u,v)$ simply as $F_k\otimes F_m$, when $u\in F_k$ and $v\in F_m$ are chosen so that $(F_k\otimes F_m)(u,v)$ has maximal width.  We note that equation \eqref{product_of_CD2} is just a one of a kind since $F_2\otimes F_3\not\cong F_5$ already. 

Our calculations, using formulas~\eqref{card_alt_s} and~\eqref{cardinality_of_tensor_product} show that
\[
|F_n\otimes F_n|<|F_{2n}|
\]
for $2<n \le 19$ but $4611858343415=|F_{20}\otimes F_{20}|>|F_{40}|=4549082342996$. 
Earlier, \cite{DKK:2012} showed that $|F_{21}\otimes  F_{21}|>  |F_{42}|$ disproving an old Fishburn's hypothesis that $F_n$ is the largest peak-pit Condorcet domain on $n$ alternatives \citep{PF:1996,GR:2008}.

\section{Conclusion and further research}

Operations over Condorcet domains are useful in many respects. The Danilov-Karzanov-Koshevoy construction is especially useful since it converts smaller peak-pit Condorcet domains into larger peak-pit domains. Fishburn's replacement scheme \citep{PF:1996} also produces larger Condorcet domains from smaller ones but without preserving peak-pittedness. Using it Fishburn proved that $f(16)>|F_{16}|$ and since he believed that $g(n)=|F_n|$ this would imply that $f(n)>g(n)$ for large $n$.  Now that we know that $g(n)>|F_n|$, the question whether or not $f(n)=g(n)$ comes to the fore. Another interesting question is to find the smallest positive integer $n$ for which $g(n)>|F_n|$.


\bibliography{cps}


\end{document}